\documentclass[12pt,leqno]{amsart}
\usepackage{verbatim,amssymb,amsfonts,latexsym,amsmath,amsthm,tikz,enumerate,mathtools,nicefrac}

\newtheorem{theorem}{Theorem}[]
\newtheorem{lemma}[theorem]{Lemma}
\newtheorem{corollary}[theorem]{Corollary}

\newtheorem*{maintheorem1}{Theorem~\ref{thm:main}}
\newtheorem*{maintheorem2}{Theorem~\ref{thm:second}}
\newtheorem*{maintheorem3}{Theorem~\ref{thm:thick}}
\newtheorem*{maintheorem4}{Theorem~\ref{thm:Gdelta}}
\theoremstyle{definition}

\newtheorem*{theorem*}{The vertex-induced graph theorem}
\theoremstyle{remark}

\newtheorem*{claim}{Claim}
\newtheorem*{sub-claim}{sub-claim}

\newcommand{\R}{\mathbb{R}}

\newcommand{\F}{\mathcal{F}}

\newcommand{\N}{\mathbb{N}}
\newcommand{\e}{\varepsilon}

\newcommand{\explicitSet}[1]{\left\lbrace #1 \right\rbrace}

\newcommand{\set}[2]{\explicitSet{#1 \colon #2}}

\newcommand{\0}{\emptyset}

\newcommand{\card}[1]{\left\lvert #1 \right\rvert}

\newcommand{\sub}{\subseteq}

\newcommand{\tr}[1]{[\![#1]\!]}

\begin{document}

\title{Which subsets of an infinite random graph look random?}
\author{Will Brian}
\address {
William R. Brian\\
Department of Mathematics\\
Baylor University\\
One Bear Place \#97328\\
Waco, TX 76798-7328}
\email{wbrian.math@gmail.com}

\maketitle

\begin{abstract}
Given a countable graph, we say a set $A$ of its vertices is \emph{universal} if it contains every countable graph as an induced subgraph, and $A$ is \emph{weakly universal} if it contains every finite graph as an induced subgraph. We show that, for almost every graph on $\mathbb N$, $(1)$ every set of positive upper density is universal, and $(2)$ every set with divergent reciprocal sums is weakly universal. We show that the second result is sharp (i.e., a random graph on $\N$ will almost surely contain non-universal sets with divergent reciprocal sums) and, more generally, that neither of these two results holds for a large class of partition regular families.
\end{abstract}

\section{Introduction}




All the graphs considered here are simple and undirected, and $\N$ denotes the natural numbers (without $0$).

This paper is about random graphs on $\N$. One may imagine forming a graph on $\N$ by a random process whereby each possible edge is included or excluded, independently of the others, with probability $\nicefrac{1}{2}$. This process was considered by Erd\H{o}s and R\'enyi in \cite{E&R}, where they proved that it almost surely results in the same graph every time (up to isomorphism). This graph is in various places called the \emph{Erd\H{o}s-R\'enyi graph}, the \emph{countable random graph}, or the \emph{Rado graph} (it was studied by Richard Rado early on in \cite{Rad}). This graph has been thoroughly studied, as has the analogous random process for finite graphs.

Of the many important properties of the countable random graph, one of the most well known is that it contains a copy of every countable graph as an induced subgraph. This is easily proved by induction, using the so-called \emph{extension property} of the countable random graph, which states that for every finite set $F$ of vertices and every $A \sub F$, there is a vertex outside of $F$ connected to everything in $A$ and nothing in $F-A$. A graph that contains a copy of every countable graph is called \emph{universal}.

Thus the countable random graph contains copies of the infinite complete graph, the infinite empty graph, and everything in between. In particular, there are subsets $A \sub \N$ such that the induced subgraph on $A$ looks highly non-random. The idea of this paper is to explore the simple question: \emph{which subsets?} More specifically, we will, for various notions of what it means to be a ``large'' subset of $\N$, determine whether we should expect large subsets of a random graph to look random.

Given a graph on $\N$, let us say that $A \sub \N$ is \emph{universal} if it contains a copy of every countable graph as an induced subgraph (equivalently, $A$ is universal if it contains a copy of the countable random graph). Let us say that $A \sub \N$ is \emph{weakly universal} if it contains a copy of every finite graph.

In Section~\ref{sec:pud}, we prove two positive results along these lines:

\begin{theorem}\label{thm:main}
For almost every graph on $\N$, every set with positive upper density is universal.
\end{theorem}

\begin{theorem}\label{thm:second}
For almost every graph on $\N$, every set $A \sub \N$ with $\sum_{n \in A}\frac{1}{n} = \infty$ is weakly universal.
\end{theorem}

In Section~\ref{sec:neg} we will prove two negative results, each stating that the conclusion of one or both of these theorems fails for a broad class of notions of largeness. One consequence is that the conclusion of Theorem~\ref{thm:second} cannot be improved from ``weakly universal'' to ``universal.'' Indeed, one gets from Section~\ref{sec:neg} the impression that the positive upper density sets are nearly alone in satisfying Theorem~\ref{thm:main}.

In order to state these negative results more precisely, recall that a \emph{Furstenberg family}, or simply a \emph{family}, is a nonempty collection $\F$ of subsets of $\N$ that is closed under taking supersets: if $A \in \F$ and $A \sub B$, then $B \in \F$. A family $\F$ is \emph{partition regular} if whenever $A \in \F$ and $A = \bigcup_{i \leq n}A_i$, then there is some $i \leq n$ with $A_i \in \F$. Intuitively, we think of partition regular families as providing a coherent notion of what it means for a set of natural numbers to be ``large.'' For example, the sets of positive upper density and the sets having divergent reciprocal sums both form partition regular families.

In Section~\ref{sec:neg} we prove the following:

\begin{theorem}\label{thm:thick}
For almost every graph on $\N$, there is a thick $A \sub \N$ with no edges. 
\end{theorem}

\begin{theorem}\label{thm:Gdelta}
Let $\F$ be any $\mathbf \Pi^0_2$ partition regular family. For almost every graph on $\N$, there is some $A \in \F$ that is not universal.
\end{theorem}

The notions of thickness and of $\mathbf \Pi^0_2$ partition regular families will be defined in Section~\ref{sec:neg}. The important thing to note here is that practically every ``naturally occurring'' partition regular family of subsets of $\N$ is covered by the hypotheses of Theorems \ref{thm:thick} and \ref{thm:Gdelta}, with the family of positive upper density sets being the only notable exception.


\section{Two positive results}\label{sec:pud}

We begin this section with a proof of Theorem \ref{thm:main}:

\begin{maintheorem1}
For almost every graph on $\mathbb N$, every set of positive upper density is universal.
\end{maintheorem1}

We approach the proof through a sequence of lemmas distilling the necessary bits of probability theory. For Theorem~\ref{thm:main} these are fairly elementary: only a few observations based on the strong law of large numbers are required. 

Before stating the first lemma we require one more definition: for a finite set of vertices $F$, a \emph{type} over $F$ (or, context permitting, simply a \emph{type}) is a way of specifying how a vertex connects to each element of $F$. Formally, a type is a predicate defined in terms of an edge relation. For example, if $F = \{a,b,c\}$, then there are eight types over $F$; one of them is the predicate ``connects to $a$ and $b$ but not to $c$'' (and the other seven are defined similarly). For a fixed graph $G$ and $F \sub G$, we could define a type as a set of vertices rather than as a predicate. However, in what follows we will want to consider a single type in several different graphs at once (different graphs on the same vertex set), and we want to be able to speak sensibly about a single type defining different sets in different graphs.

Our first lemma is a fairly straightforward consequence of the strong law of large numbers. It can be thought of as a stronger version of the extension property of the countable random graph mentioned in the introduction.

\begin{lemma}\label{lem:stronglaw}
Almost every graph on $\N$ has the following property:
\begin{itemize}
\item[$(*)$] Let $F_1, \dots, F_n$ be pairwise disjoint subsets of $\N$, each of size $k$, and for each $i \leq n$ let $t_i$ be some fixed type over $F_i$. The set of all vertices that are not of type $t_i$ for any $i \leq n$ is a set of density $\left(1 - \frac{1}{2^k}\right)^n$.
\end{itemize}
\end{lemma}
\begin{proof}
There are only countably many finite families $F_1,\dots,F_n$ of pairwise disjoint size-$k$ subsets of $\N$. By the countable additivity of probability, it suffices to prove the conclusion of the lemma for a single, fixed collection $F_1, \dots, F_n$ of pairwise disjoint size-$k$ subsets of $\N$.

Let $F_1, \dots, F_n$ be pairwise disjoint subsets of $\N$, each of size $k$, and for each $i \leq n$ let $t_i$ be some fixed type over $F_i$. Fix $m \in \N - \bigcup_{i \leq n}F_i$. An easy computation shows that the probability that $n$ is not of type $t_i$ for any $i$ is $\left(1 - \frac{1}{2^k}\right)^n$. The desired conclusion follows from the strong law of large numbers.
\end{proof}

\begin{lemma}\label{lem:recursionengine}
Almost every graph on $\N$ has the following property:
\begin{itemize}
\item[$(\dagger)$] Let $F$ be a finite subset of $\N$, and let $A \sub \N$. Suppose that for every type $t$ over $F$,
$$\set{m \in A - F}{m \text{ has type }t}$$
has positive upper density. Then for all but finitely many $n \in \N$, for every type $t'$ over $F \cup \{n\}$ the set
$$\set{m \in A - (F \cup \{n\})}{m \text{ has type }t'}$$
has positive upper density.
\end{itemize}
\end{lemma}
\begin{proof}
It suffices to show that every graph satisfying property $(*)$ from Lemma~\ref{lem:stronglaw} also satisfies property $(\dagger)$. 

Suppose some graph on $\N$ satisfies $(*)$. Fix $F \sub \N$ with $|F| = k$, let $t_1,t_2,\dots,t_{2^k}$ denote all the different types over $F$, and let $A \sub \N$ have the property that
$$A_i = \set{m \in A - F}{m \text{ has type }t_i}$$
has positive upper density for every $i \leq 2^k$. Pick $\ell \in \N$ large enough that $\frac{1}{2^{\ell+1}}$ is (strictly) less than the minimum of the upper densities of the $A_i$.

\begin{claim}
For any $i \leq 2^k$, there are at most $\ell$ values of $n \in \N - F$ such that the set
$$C_i(n) = \set{m \in A_i - (F \cup \{n\})}{m \text{ connects to } n}$$
does not have positive upper density.
\end{claim}

\begin{proof}[Proof of claim]
Suppose otherwise. Then there is a set $L$ of size $\ell+1$ such that for each $n \in L$, $C_i(n)$ has density $0$. For each $n \in L$, let
$$D(n) = \set{m \in \N - (F \cup \{n\})}{m \text{ does not connect to } n}.$$
By $(*)$, the density of $D_L = \bigcap_{n \in L}D(n)$ is $\left(1 - \frac{1}{2}\right)^{\ell+1} = \frac{1}{2^{\ell+1}}$. However,
$$A_i \sub D_L \cup \bigcup_{n \in L}C_i(n).$$
Because each $C_i(n)$ has density $0$, this implies that $A_i$ has positive upper density at most the density of $D_L$, namely $\frac{1}{2^{\ell+1}}$. This contradicts our choice of $\ell$.
\end{proof}

Interchanging the roles of ``connects'' and ``does not connect,'' the same argument can be used to show that, for each $i \leq 2^k$, there are at most $\ell$ values of $n \in \N - F$ such that the set
$$D_i(n) = \set{m \in A_i - (F \cup \{n\})}{m \text{ does not connect to } n}$$
fails to have positive upper density.

By this claim, for each $i \leq 2^k$ there are at most $2 \ell$ values of $n$ that make one of the sets $C_i(n)$ or $D_i(n)$ fail to have positive upper density. Thus there are at most $2^{k+1} \ell$ values of $n$ such that one of $C_i(n)$ or $D_i(n)$ has density $0$ for some $i \leq 2^k$.

Fix some $n$ not in this finite set: i.e., some $n$ such that every $C_i(n)$ and $D_i(n)$, for $i \leq 2^k$, has positive upper density. 

Let $t$ be some type over $F \cup \{n\}$. There is some $i \leq 2^k$ such that ``$m$ has type $t$ over $F \cup \{n\}$'' is either the assertion ``$m$ has type $t_i$ over $F$ and $m$ connects to $n$'' or the assertion ``$m$ has type $t_i$ over $F$ and $m$ does not connect to $n$.'' In other words,
$$\set{m \in A - (F \cup \{n\})}{m \text{ has type }t}$$
is either $C_i(n)$ or $D_i(n)$. Our choice of $n$ guarantees that, either way, this set has positive upper density.
\end{proof}

\begin{proof}[Proof of Theorem~\ref{thm:main}]
By Lemma~\ref{lem:recursionengine}, it suffices to show that if a graph on $\N$ satisfies $(\dagger)$ then every $A \sub \N$ of positive upper density contains every countable graph as an induced subgraph. Suppose we have a graph on $\N$ satisfying $(\dagger)$, let $A \sub \N$ have positive upper density, and let $G$ be a countable graph. Let $\set{v_n}{n \in \N}$ be an enumeration of the vertices of $G$.

We will use recursion to find a sequence $a_1,a_2,\dots,a_n,\dots$ of points in $A$ such that the map $v_n \mapsto a_n$ is an isomorphism from $G$ to the induced subgraph on $\set{a_n}{n \in \N}$. Pick $a_1 \in A$ so that, for each of the two types $t_i$ ($i = 1,2$) over $\{a_1\}$,
$$A_i = \set{m \in A - \{a_1\}}{m \text{ has type }t_i}$$
has positive upper density. Some such $a_1$ exists by $(\dagger)$ (setting $F = \0$).

Assume now that $a_1,a_2,\dots,a_{n-1}$ have all been chosen in such a way that the following two inductive hypotheses are satisfied:
\begin{itemize}
\item the map $v_j \mapsto a_j$, $j < n$, is an isomorphism of induced subgraphs.
\item for each type $t$ over $\{a_1,a_2,\dots,a_{n-1}\}$, $\set{m \in A}{m \text{ has type }t}$ has positive upper density.
\end{itemize}
By the first hypothesis, there is some type $t^G_{n-1}$ over $\{a_1,a_2,\dots,a_{n-1}\}$ such that any $m \in \N - \{a_1,a_2,\dots,a_{n-1}\}$ of type $t^G_{n-1}$ will have the property that the map
$$v_j \mapsto a_j \text{ for } j < n \text{, and } v_n \mapsto m$$
is an isomorphism of induced subgraphs. Using $(\dagger)$, we may find some $a_n \in A$ of type $t_{n-1}^G$ such that, for every type $t$ over $\{a_1,a_2,\dots,a_n\}$,
$$\set{m \in A - \{a_1,a_2,\dots,a_n\}}{m \text{ has type }t}$$
has positive upper density. This choice of $a_n$ preserves both inductive hypotheses for the next step of the recursion.

Thus we may construct an infinite sequence $a_1,a_2,\dots,a_n,\dots$ of points in $A$ such that for each $n$ the map $v_i \mapsto a_i$, $i \leq n$, is an isomorphism of induced subgraphs. It follows that the map $v_n \mapsto a_n$ is an isomorphism from $G$ to the induced subgraph on $\set{a_n}{n \in \N}$.
\end{proof}

We now move to the proof of Theorem~\ref{thm:second}. For convenience, let us say that $A \sub \N$ is \emph{substantial} whenever $\sum_{n \in A}\frac{1}{n} = \infty$.

\begin{maintheorem2}
For almost every graph on $\N$, every substantial set is weakly universal.
\end{maintheorem2}

Again we will begin the proof with a few lemmas. If $\varphi(H)$ expresses a property of a (variable) graph $H$, we define
$$P_n(\varphi(H)) = \frac{1}{2^{n \choose 2}} \card{\set{H}{H \text{ is a graph on } \{1,2,\dots,n\} \text{ and }\varphi(H)}}$$
More colloquially, $P_n(\varphi(G))$ is the probability that a randomly chosen graph on $n$ (labelled) vertices has property $\varphi$. 

A graph $H$ is called \emph{$G$-free} if it does not contain $G$ as an induced subgraph.

\begin{lemma}[Janson, \L{}uczak, and Ruci\'nski]\label{lem:jlr}
Let $G$ be a finite graph. There is a positive constant $c$ such that
$$P_n(H \ \mathrm{is} \ G\text{-}\mathrm{free}) \leq 2^{-cn^2}$$
for sufficiently large $n$.
\end{lemma}
\begin{proof}
Let $G$ be a finite graph. If $G$ has no edges then this result is fairly routine (and in fact much stronger bounds have been found; see, e.g., \cite{Frieze}), and in any case the result for $G$ follows by symmetry from the result for the complement of $G$. Thus we may assume that $G$ has an edge.

In \cite{jlr}, Janson, \L{}uczak, and Ruci\'nski prove that (when $G$ has an edge) there is a positive constant $c_0$ such that, for all $n$,
$$P_n(H \ \mathrm{is} \ G\text{-}\mathrm{free}) \leq 2^{-c_0M}, \text{ where}$$
$$M = \min \set{\frac{n^v}{2^{e}}}{G \text{ has a subgraph with }v \text{ vertices and } e > 0 \text{ edges}}.$$
When $G$ is fixed, we will have $M = \frac{1}{2}n^2$ for all sufficiently large $n$. Setting $c = \frac{1}{2}c_0$ proves the lemma.
\end{proof}

\begin{lemma}\label{lem:finitegraphs}
Let $G$ be a finite graph, and let $\varphi(H,G,k)$ abbreviate the statement that $H$ contains a $G$-free induced subgraph of size $k$. There is a fixed natural number $N$ such that, if $f: \N \to \N$ is the function $n \mapsto \lceil N \log_2 n \rceil$, then
$$P_n(\varphi(H,G,f(n))) < n^{-2f(n)}$$
for all sufficiently large $n$.
\end{lemma}
\noindent \emph{Proof. }Fix a finite graph $G$. By Lemma~\ref{lem:jlr}, there is a positive constant $c$ and a natural number $M_0$ such that
$$P_n(H \ \mathrm{is} \ G\text{-}\mathrm{free}) \leq 2^{-cn^2}$$
for all $n \geq M_0$.

Let $N$ be any number larger than $\nicefrac{3}{c}$, and let $f: \N \to \N$ be the function $n \mapsto \lceil N \log_2 n \rceil$. Note that $\lim_{n \to \infty}f(n) = \infty$; in other words, there is some $M$ such that for all $n \geq M$, $f(n) \geq M_0$. 

Suppose $n \geq M$, and let $k = f(n)$. For each set $S$ of size $k$, there is a probability of $2^{-ck^2}$ that a randomly chosen graph on $S$ will be $G$-free. Thus the probability that a graph on $\{1,2,\dots,n\}$ has a $k$-sized $G$-free induced subgraph is at most ${n \choose k} 2^{-ck^2}$, which gives:
\begin{align*}
\pushQED{\qed} 
P_n(\varphi(H,G,k)) &\leq {n \choose k} 2^{-ck^2} \leq n^k \cdot 2^{-ck^2} = n^k(2^{c \lceil N \log_2 n \rceil})^{-k} \\
&\leq n^k (2^{\log_2 n})^{-cNk} \leq n^k \cdot n^{-3k} = n^{-2k}. \qedhere
\popQED
\end{align*}

\begin{lemma}\label{lem:finite}
Let $G$ be a finite graph. There is a natural number $N$ (depending only on $G$) such that almost every graph on $\N$ has the following property:
\begin{itemize}
\item[$(\ddagger)$] There is some $m \in \N$ such that for all $k \geq m$, no interval of the form $[2^k,2^{k+1})$ contains a $G$-free graph of size $kN$.
\end{itemize}
\end{lemma}
\begin{proof}
Let $N$ be the number guaranteed by Lemma~\ref{lem:finitegraphs}. For any given $k \in \N$, let $P_k$ denote the probability that the interval $[2^k,2^{k+1})$ contains a $G$-free graph of size $kN$. By Lemma~\ref{lem:finitegraphs} (setting $n=2^k$, which makes $f(n) = \lceil N \log_2 2^k \rceil = kN$), if $k$ is sufficiently large then
$$P_k < (2^k)^{-2(kN)} = 2^{(1-2N)k}.$$
Because $N \in \N$, we have $1-2N = -a$ for some $a \in \N$, and the $P_k$ are bounded above by the geometric sequence $2^{-ak}$.

Let $\varphi(H,G,m)$ denote the statement that, for some $k \geq m$, the interval $[2^k,2^{k+1})$ contains a $G$-free graph of size $kN$. In other words, $\varphi(H,G,m)$ is the statement that $(\ddagger)$ fails at $m$. We have
$$P_n(\varphi(H,G,m)) = \sum_{k \geq m}P_k \leq \sum_{k \geq m}2^{-ak} = \frac{2^{1-m}}{2^a-1}.$$
Thus $P_n(\varphi(H,G,m))$ approaches $0$ as $m$ grows large. Therefore, almost surely, a random graph on $\N$ will fail to satisfy $\varphi(H,G,m)$ for some $m \in \N$. It follows that almost every graph on $\N$ satisfies $(\ddagger)$.
\end{proof}

\begin{proof}[Proof of Theorem~\ref{thm:second}]
We will prove the contrapositive: \emph{if $A \sub \N$ fails to contain some finite graph $G$ as an induced subgraph, then $A$ is not substantial}. There are only countably many finite graphs, so, by the countable additivity of probability, it suffices to prove for any particular finite graph $G$ that, for almost every graph on $\N$, every $G$-free $A \sub \N$ fails to be substantial.

Let $G$ be a fixed finite graph. By Lemma~\ref{lem:finite}, it suffices to show that any graph on $\N$ with property $(\ddagger)$ has no substantial $G$-free sets. Suppose we have such a graph, and let $A \sub \N$ be $G$-free. Let $m$ and $N$ be as in the statement of $(\ddagger)$. If $A$ is $G$-free, then so is every subset of $A$, in particular those of the form $A \cap [2^k,2^{k+1})$. Applying $(\ddagger)$, we have
$$|A \cap [2^k,2^{k+1})| \leq kN$$
for every $k \geq m$. Therefore
$$\sum_{n \in A \cap [2^k,2^{k+1})}\frac{1}{n} \,\leq\, \frac{kN}{2^k}$$
for every $k \geq m$. It follows that
$$\sum_{n \in A}\frac{1}{n} \,\leq\, \sum_{n < 2^m}\frac{1}{n} + \sum_{k \geq m}\frac{kN}{2^k}.$$
This sum converges, so $A$ is not substantial.
\end{proof}

Let $f: \N \to \R^+$ be some function with the property that $\sum_{n \in \N}\frac{1}{n} = \infty$ (a \emph{weight function}), and define $A \sub \N$ to be $f$-substantial provided that $\sum_{n \in A}\frac{1}{n} = \infty$. The $f$-substantial sets form a partition regular family. If $f$ goes to $0$ sufficiently quickly (e.g., if $f(n) = \nicefrac{1}{n^\e}$ for some fixed $\e > 0$), then the conclusion of Theorem~\ref{thm:second} still holds for the family of $f$-substantial sets. This follows from the given proof for Theorem~\ref{thm:second}, by appropriately modifying the last five lines.

\section{Two negative results}\label{sec:neg}

In this section we will prove two theorems, each stating that a broad class of partition regular families fails, almost surely, to satisfy the conclusion of Theorem \ref{thm:main} or \ref{thm:second}. We will begin with the easier-to-prove of these two theorems:

\begin{maintheorem3}
Let $G$ be any countable graph. For almost every graph on $\N$, there is a thick $A \sub \N$ with no edges. 
\end{maintheorem3}

Recall that $A \sub \N$ is \emph{thick} if it contains arbitrarily long intervals. The family of thick sets is not partition regular, but it is contained in many important partition regular families. This allows us to deduce from Theorem~\ref{thm:thick} the following corollary:

\begin{corollary}\label{cor:thick}
Let $\F$ be any of the following families:
\begin{enumerate}
\item sets containing arbitrarily long arithmetic progressions.
\item (more generally) sets satisfying the conclusion of the polynomial van der Waerden theorem.
\item the piecewise syndetic sets.
\item sets of positive upper Banach density.
\item the IP-sets (i.e., sets satisfying the conclusion of Hindman's Theorem).
\item the $\Delta$-sets (i.e., sets containing $\{s_j-s_i : i,j \in \mathbb N, i < j\}$ for some infinite sequence $\langle s_i : i \in \mathbb N \rangle$ of natural numbers).
\item the central sets (i.e., sets belonging to some minimal idempotent ultrafilter).
\item sets containing infinitely many solutions to some particular partition regular system of linear equations (e.g., containing infinitely many solutions to the equation $x+y=z$).
\end{enumerate}
For almost every graph on $\N$, there is some $A \in \F$ containing no edges. In particular, there are sets in $\F$ that are not weakly universal.
\end{corollary}
\begin{proof}
If a set is thick, then it is in every one of these eight families. This is obvious for $(1)$ - $(4)$, and is easily proved by recursion for $(5)$ and $(6)$. For $(7)$, see \cite{H&S}, Theorem 4.48. For $(8)$, see chapter 15 of \cite{H&S}.
\end{proof}

\begin{lemma}
Almost every graph on $\N$ has the following property:
\begin{itemize}
\item[$(\ddagger \hspace{-2.57mm} -\!)$] For every finite $F \sub \N$ and every $n \in \N$, there is an interval $[k,k+n]$ (disjoint from $F$) such that there are no edges between any member of $F$ and any member of $[k,k+n]$, and there are no edges between the members of $[k,k+n]$.
\end{itemize}
\end{lemma}
\begin{proof}
There are only countably many pairs $(F,n)$, where $F$ is a finite set of natural numbers and $n \in \N$. By the countable additivity of probability, it suffices to prove that the conclusion of $(\ddagger \hspace{-2.57mm} -\!)$ holds for a single, fixed finite $F \sub \N$ and some fixed $n \in \N$.

Given $F$ and $n$, it is clear that, for every $k > \max F$, the probability that no member of $[k,k+n]$ connects to any member of $F$ and that no two members of $[k,k+n]$ are connected is
$$\left( \frac{1}{2^{|F|}} \right)^{n+1} \cdot \frac{1}{2^{n+1 \choose 2}}.$$
In particular, the probability is positive and indenpenent of $k$. For $k$ in the infinite set $(n+1)\N$, these probabilities are also independent of each other. Therefore, almost surely, some $k$ must have this property.
\end{proof}

\begin{proof}[Proof of Theorem~\ref{thm:thick}]
Suppose we have a graph on $\N$ satisfying $(\ddagger \hspace{-2.57mm} -\!)$. We will construct a sequence of intervals by recursion. Begin by setting $I_1 = \{1\}$; then, given $I_1,I_2, \dots, I_{n-1}$, use $(\ddagger \hspace{-2.57mm} -\!)$ to find an interval $I_n$ of length $n$ such that no members of $I_n$ are connected to each other or to any member of $\bigcup_{j < n}I_j$. This defines a sequence $I_1,I_2,\dots$ of intervals, and the thick set $\bigcup_{j \in \N}I_j$ has no edges.
\end{proof}

We leave it as an exercise to show that the proof of Theorem~\ref{thm:thick} can be modified to show that, given any countable graph $G$, almost every graph on $\N$ contains a thick set $A \sub \N$ such that the induced subgraph on $A$ is isomorphic to $G$. It follows that, if $\F$ is any of the partition regular families mentioned in Corollary~\ref{cor:thick}, then there are no restrictions on which graphs will (almost surely) be realized as induced subgraphs on members of $\F$.






Our next theorem states that every partition regular family with a sufficiently simple definition will (almost surely) contain sets that are not universal. Here ``sufficiently simple'' means $\mathbf \Pi^0_2$, or $G_\delta$: a countable intersection of open sets.

Recall that the power set of $\N$, $\mathcal P(\N)$, has a natural topological structure (it can be naturally identified via characteristic functions with the Cantor space $2^\N$). The topology on $\mathcal P(\N)$ is the topology of finite agreement: basic open subsets of $\mathcal P(\N)$ are of the form
$$\tr{F,n} = \set{A \sub \N}{A \cap [1,n] = F}$$
where $F \sub [1,n]$. Roughly this means that $U \sub \mathcal P(\N)$ is open if for every $A \in U$, it can be determined that $A \in U$ by looking only at $A \cap [1,n]$ for some sufficiently large $n$. In other words, membership in $U$ is a condition satisfiable in finite time. By extension, membership in a $G_\delta$ subset of $\mathcal P(\N)$ can be thought of as a conjunction of countably many conditions, each satisfiable in finite time.

For example, the family of substantial sets is $\mathbf \Pi^0_2$ because it is the intersection of the countably many open sets
$$\textstyle V_n = \set{A \sub \N}{\sum_{m \in A}\frac{1}{m} > n}.$$
Other examples include the family of thick sets, or the families mentioned in parts $(1)$, $(2)$, and $(8)$ of Corollary~\ref{cor:thick}.

\begin{maintheorem4}
Let $\F$ be any $\mathbf \Pi^0_2$ partition regular family. For almost every graph on $\N$, there is some $A \in \F$ that has finite connected components. In particular, not every set in $\F$ is universal.
\end{maintheorem4}

\begin{corollary}
Let $\F$ be any of the following partition regular families:
\begin{enumerate}
\item the substantial sets.
\item (more generally) for any weight function $f: \N \to \R^+$ with the property that $\sum_{n \in \N}f(n)$ diverges, the family of all sets $A \sub \N$ such that $\sum_{n \in A}f(n)$ diverges.
\end{enumerate}
For almost every graph on $\N$, there is a set in $\F$ that is not universal.
\end{corollary}

Theorem~\ref{thm:Gdelta} implies that the conclusion of Theorem~\ref{thm:second} cannot be strengthened from ``almost universal'' to ``universal.'' In fact, it tells us more: (almost surely) a graph on $\N$ will have a substantial set with finite connected components. Thus Theorem~\ref{thm:second} has the strongest conclusion that could be hoped for.

Recall that a set $\mathcal X$ of subsets of $\N$ is called a \emph{tail set} if for every $A \in \mathcal X$ and every finite set $B$, $A \Delta B \in \mathcal X$; in other words, a tail set is a collection of subsets of $\N$ that is closed under finite modifications.

\begin{lemma}
Let $\F$ be a partition regular family containing only infinite sets. Then $\F$ is a tail set.
\end{lemma}
\begin{proof}
Let $A \in \F$, and let $B$ be a finite set. Because
$$A \subseteq (A \cap [1,\max B]) \cup (A \Delta B),$$
partition regularity implies that one of either $A \cap [1,\max B]$ or $B \Delta A$ is in $\F$. $A \cap [1,\max B]$ is finite, so we must have $A \Delta B \in \F$.
\end{proof}

For any $0 < p < 1$, we may define a probability measure $\mu_p$ on $\mathcal P(\N)$ by asserting that every subbasic open set of the form
$$\set{A \sub \N}{n \in A}$$
has measure $p$, and its complement
$$\set{A \sub \N}{n \notin A}$$
has measure $1-p$.

\begin{lemma}\label{lem:tailset}
Let $\F$ be a partition regular family containing only infinite sets. For every $0 < p < 1$, $\mu_p(\F)=1$.
\end{lemma}
\begin{proof}
Fix $0 < p < 1$. By Kolmogorov's zero-one law, the $\mu_p$-measure of any tail set in $2^\N$ is either $0$ or $1$. We will show that $\F$ cannot have $\mu_p$-measure $0$. To see this, consider the map $c: \mathcal P(\N) \rightarrow \mathcal P(\N)$ sending every set to its complement. This map changes the measure of basic open sets by at most a factor of $\max \{\frac{1-p}{p},\frac{p}{1-p}\}$; therefore the same is true for all measurable sets and, in particular, $c$ sends null sets to null sets. Thus, if $\F$ is null, $\F \cup c[\F]$ is null, and there is some $A \sub \N$ such that $A \notin \F \cup c[\F]$. But $A \notin c[\F]$ if and only if $c(A) \notin \F$, so we have neither $A$ nor $c(A)$ in $\F$. This contradicts the partition regularity of $\F$, so $\F$ has measure $1$.
\end{proof}

\begin{lemma}
Let $\F$ be a partition regular family containing only infinite sets. Almost every graph on $\N$ has the following property:
\begin{itemize}
\item[$(\dagger \! \dagger)$] For every finite $F \sub \N$ and every type $t$ over $F$,
$$\set{n > \max F}{n \text{ has type }t\text{ over }F} \in \F.$$
\end{itemize}
\end{lemma}
\begin{proof}
As usual, by the countable additivity of probability it suffices to prove the claim for a single fixed finite $F \sub \N$ and a single type $t$ over $F$.

The set of all vertices having type $t$ over $F$ is determined randomly: the probability is $\frac{1}{2^{|F|}}$ that any particular $n > \max F$ has type $t$ over $F$, independently of whether any other vertex has type $t$ over $F$. Therefore the set of all vertices $n > \max F$ having type $t$ over $F$ is a point in the space $\mathcal P(\N \setminus [1,\max F])$, chosen randomly according to the probability measure $\mu_p$, where $p = \frac{1}{2^{|F|}}$. By the previous lemma, the set of all vertices $n > \max F$ having type $t$ over $F$ is in $\F$.
\end{proof}

\begin{proof}[Proof of Theorem~\ref{thm:Gdelta}]
Let $\F$ be a $\mathbf \Pi^0_2$ partition regular family of sets. Let us assume that $\F$ contains only infinite sets (otherwise the theorem is trivially true for $\F$). Let $U_1,U_2,\dots$ be open subsets of $2^\N$ such that $\F = \bigcap_{n \in \N}U_n$. It suffices to show that if a graph on $\N$ satisfies $(\dagger \! \dagger)$, then there is some $A \in \F$ that is not universal.

Suppose we have such a graph. We will construct $A \in \F$ by recursion with the property that every connected component of $A$ is finite. At stage $n$ of the recursion, we will add finitely many points to $A$ that will ensure $A$ is in $U_n$, and at the same time will be disconnected from all the points already in $A$.

For convenience, set $k_0=0$. To begin the recursion, pick a finite set $F_1$ and $k_1 \in \N$, such that $F_1 \sub [1,k_1]$ and
$$\tr{F_1,k_1} \sub U_1$$
(there must be some such $F_1$ and $k_1$ because $U_1$ is open and nonempty). For the recursive step, suppose finite sets $F_1, F_2, \dots, F_{n-1}$ and natural numbers $k_1,k_2,\dots,k_{n-1}$ have been chosen already, and that they satisfy the following hypotheses:
\begin{itemize}
\item $k_1 < k_2 < \dots < k_{n-1}$.
\item for each $i < n$, $F_i \sub (k_{i-1},k_i]$.
\item for each $i < n$,
$$\tr{F_1 \cup F_2 \cup \dots \cup F_i,k_i} \sub U_i.$$
\item for each $i < j < n$, no member of $F_i$ connects to any member of $F_j$.
\end{itemize}
We will use $(\dagger \! \dagger)$ to find $k_n$ and $F_n$. Let $t$ denote the type over $[1,k_{n-1}]$ stating that a vertex does not connect to anything in $[1,k_{n-1}]$. By $(\dagger \! \dagger)$, the set
$$T = \set{n > k_{n-1}}{n \text{ has type } t \text{ over }[1,k_{n-1}]}$$
is in $\F$. Because $\F$ is closed under taking supersets,
$$T' = T \cup F_1 \cup F_2 \dots \cup F_{n-1} \in \F.$$
In particular, $T' \in U_n$. Because $U_n$ is open, there is some basic open subset $\tr{F,k}$ of $\mathcal P(\N)$ such that
$$T' \in \tr{F,k} \sub U_n.$$
Choose any $k_n > \max\{k,k_{n-1}\}$ and let $F_n = T' \cap (k_{n-1},k_n]$. Note that 
$$T' \cap [1,k_n] = F_1 \cup F_2 \cup \dots \cup F_n$$
because $\max (F_1 \cup F_2 \cup \dots \cup F_{n-1}) \leq k_{n-1} < \min T$. Therefore
$$\tr{F_1 \cup F_2 \cup \dots \cup F_n,k_n} \sub \tr{k,F} \sub U_n.$$
By our choice of $F_n$, no vertex of $F_n$ connects to any vertex of $F_i$ for any $i < n$. Thus all of our recursive hypotheses are still satisfied at $n$, and this completes the recursion.

Thus we obtain a sequence $F_1,F_2,F_3,\dots$ of finite sets and a sequence $k_1 < k_2 < \dots$ of natural numbers such that
\begin{itemize}
\item for each $n \in \N$, $F_n \sub (k_{n-1},k_n]$.
\item for each $n \in \N$,
$$\tr{F_1 \cup F_2 \cup \dots \cup F_n,k_n} \sub U_n.$$
\item for any $m < n \in \N$, no member of $F_m$ connects to any member of $F_n$.
\end{itemize}
Let $A = \bigcup_{n \in \N}F_n$. By design, we have $A \in \tr{F_1 \cup F_2 \cup \dots \cup F_n,k_n} \sub U_n$ for every $n$, so that $A \in \bigcap_{n \in \N}U_n = \F$. On the other hand, it is clear that every connected component of $A$ is finite.
\end{proof}

Let us say that a partition regular family is \emph{nontrivial} if it contains only infinite sets. By a proof very similar to that of Lemma~\ref{lem:tailset}, one may show that every nontrivial partition regular family is co-meager. Also, any such family $\F$ has dense complement in $\mathcal P(\N)$ (because the set of finite subsets of $\N$ is dense in $\mathcal P(\N)$); in particular, any closed subset of $\F$ is nowhere dense. Applying the Baire Category Theorem, it follows that there are no nontrivial $\mathbf \Sigma^0_2$ partition regular families. Combining this observation with Theorems \ref{thm:main}, \ref{thm:second}, and \ref{thm:Gdelta}, the following picture emerges:

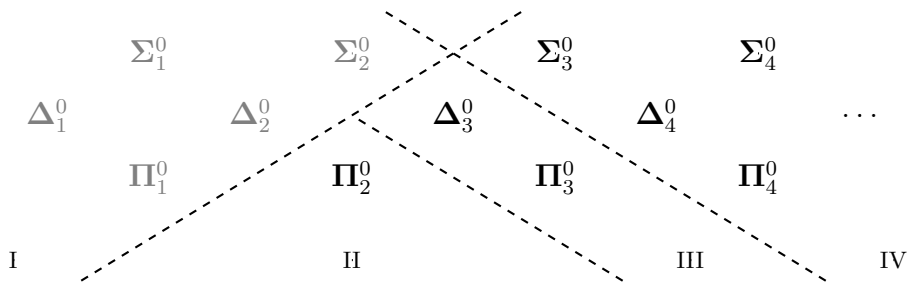
\begin{center}
\begin{tikzpicture}[style=thick, xscale=.45,yscale=.55]

\draw[gray] (0,0) node {$\mathbf \Delta^0_1$} -- (0,0);
\draw[gray] (3,1.5) node {$\mathbf \Sigma^0_1$} -- (3,1.5);
\draw[gray] (3,-1.5) node {$\mathbf \Pi^0_1$} -- (3,-1.5);

\draw[gray] (6,0) node {$\mathbf \Delta^0_2$} -- (6,0);
\draw[gray] (9,1.5) node {$\mathbf \Sigma^0_2$} -- (9,1.5);
\draw (9,-1.5) node {$\mathbf \Pi^0_2$} -- (9,-1.5);

\draw (12,0) node {$\mathbf \Delta^0_3$} -- (12,0);
\draw (15,1.5) node {$\mathbf \Sigma^0_3$} -- (15,1.5);
\draw (15,-1.5) node {$\mathbf \Pi^0_3$} -- (15,-1.5);

\draw (18,0) node {$\mathbf \Delta^0_4$} -- (18,0);
\draw (21,1.5) node {$\mathbf \Sigma^0_4$} -- (21,1.5);
\draw (21,-1.5) node {$\mathbf \Pi^0_4$} -- (21,-1.5);

\draw (24,0) node {$\mathbf \dots$} -- (24,0);

\draw[dashed] (14,2.5) -- (1,-4);
\draw[dashed] (23,-4) -- (10,2.5);
\draw[dashed] (17,-4) -- (9,0);

\draw (-1,-3.5) node {\scriptsize I} -- (-1,-3.5);
\draw (9,-3.5) node {\scriptsize II} -- (9,-3.5);
\draw (19,-3.5) node {\scriptsize III} -- (19,-3.5);
\draw (25,-3.5) node {\scriptsize IV} -- (25,-3.5);

\end{tikzpicture}\end{center}

In region I, there are no nontrivial partition regular families. In region IV lies the family of positive upper density sets (which is $\mathbf \Sigma^0_3$) so it is possible for families in this region to satisfy the conclusion of Theorem~\ref{thm:main}. A family in region II does not satisfy the conclusion of Theorem~\ref{thm:main} (by Theorem~\ref{thm:Gdelta}), but it may satisfy the weaker conclusion of Theorem~\ref{thm:second} (e.g., the family of substantial sets). In region III it seems that the results of this paper have nothing to say. We leave it as an open question whether it is possible for a partition regular $\mathbf \Delta^0_3$ or $\mathbf \Pi^0_3$ family to satisfy the conclusion of Theorem~\ref{thm:main}. For that matter, we leave it as an open question whether there is a partition regular $\mathbf \Delta^0_3$ or $\mathbf \Pi^0_3$ family that is not also $\mathbf \Pi^0_2$.


\begin{thebibliography}{99}
\bibitem{E&R} P. Erd\H{o}s and A. R\'enyi, ``Asymmetric graphs,'' \emph{Acta Mathematica Academiae Scientiarum Hungaricae} \textbf{14} (1963), pp. 295-315.
\bibitem{Frieze} A. M. Frieze, ``On the independence number of random graphs,'' \emph{Discrete Mathematics} \textbf{81} (1990), pp. 171-175.
\bibitem{H&S} N. Hindman and D. Strauss, \emph{Algebra in the Stone-\v{C}ech compactification}, De Gruyter, Berlin, 1998.
\bibitem{jlr} S. Janson, T. \L{}uczak, and A. Ruci\'nski, ``An exponential bound for the probability of nonexistence of a specified subgraph in a random graph,'' in \emph{Random Graphs}, ed. M. Karo\'nski, J. Jaworski, and A. Ruci\'nski (1987), pp. 79-83.
\bibitem{Rad} R. Rado, ``Universal graphs and universal functions,'' \emph{Acta Arithmetica} \textbf{9} (1964), pp. 331-340.
\end{thebibliography}
\end{document}